\documentclass[12pt]{amsart}

\usepackage{amscd}
\usepackage{verbatim}

\usepackage{caption}
\usepackage[labelformat=simple,labelfont={}]{subcaption}

\advance\hoffset by -0.5 truecm

\usepackage{amssymb}
\newtheorem{Theorem}{Theorem}[section]

\newtheorem{Lemma}[Theorem]{Lemma}

\advance\hoffset by -0.5 truecm

\usepackage{subcaption}
\usepackage{amssymb}
\usepackage{graphicx}
\usepackage{xcolor}
\usepackage{enumerate}

\newtheorem{Remark}[Theorem]{Remark}

\def\V{\mbox{Var}}

\def\R\re

\def\V{\bf V}

\def \re{{\mathbb R}}

\def \0{\lambda_{0}}

\begin{document}
\hyphenation{Ya-ma-be co-rres-pon-ding hy-po-the-sis iso-pe-ri-me-tric gen-er-at-ed man-i-folds me-tric differ-en-tial in-va-riants}






\title[Isoperimetric estimates in the product of  small and large volume]{Isoperimetric estimates in the product of  small and large volume manifolds}



\author[J. M. Ruiz]{Juan Miguel Ruiz$^{\dagger*}$}
 \thanks{$^\dagger$ ENES  UNAM. Departamento  de Matem\'aticas. Le\'on, Gto., M\'exico. mruiz@enes.unam.mx}


\author[A. V. Juarez]{Areli V\'azquez Ju\'arez$^{\ddagger }$}
\thanks{$^\ddagger$ ENES UNAM.  Departamento de Matem\'aticas. Le\'on, Gto., M\'exico. areli@enes.unam.mx.}
\thanks{* Corresponding author.}



\begin{abstract}

	 Let $(M^m,g)$, $(N^n,h)$ be closed Riemannian manifolds,  $m,n\geq 2$, with concave isoperimetric profiles and volumes $V_M$, $V_N$ respectively. We consider a one parameter family of product manifolds of the same volume, $(X,G_{\lambda})=(M^m\times N^n,\lambda^{2n}g+ \lambda^{-2m}h)$, $\lambda>0$, and  estimate a lower bound for their isoperimetric profile for large $\lambda$. In particular, we show that for $\alpha \in (0,1)$ and $v_0 \in (0, V_MV_N)$, one can estimate $\lambda_{0}>0$, such that for $\lambda>\lambda_0$, 
	 $$	(2\alpha-1)^{2} f_{M,\lambda}(v_0) \leq	I_{(X,G_{\lambda})}(v_0)\leq  f_{M,\lambda}(v_0)$$

	\noindent	where  $f_{M,\lambda}(v)= \lambda^{-n} V_N I_{(M,g)}(\frac{v}{V_N})$ and $	I_{(M,g)}$ is the isoperimetric profile of $(M,g)$.

	Moreover,  	if $(M,g)=(S^2,g_0)$,  the $2-$sphere with the round metric, and $(N,h)=(T^2,h)$, the 2-torus with the flat metric $h$, then, in this setting, we provide an interval $(v_0,V_MV_N-v_0)$ of volumes and $\lambda^*>0$, such that, for $\lambda>\lambda^*$, regions of the type  ${ D^{\lambda}_R\times T^2_{\lambda} }$, with    $Vol({ D^{\lambda}_R\times T^2_{\lambda} })\in (v_0,V_MV_N-v_0)$, are actual isoperimetric regions in $ (S^2\times T^2,\lambda^{2}g_0+ \lambda^{-2}h)$; being $D^{\lambda}_R$ a disk in $(S^2,\lambda^{2}g_0)$ and $T^2_{\lambda}=(T^2, \lambda^{-2}h)$.

\end{abstract}

\maketitle

\section{Introduction}

Let  $(M^m,g)$ be a Riemannian manifold of volume $V_M$. The isoperimetric  profile of  $(M,g)$ is the function $I_{(M,g)} : (0,V_M ) \rightarrow [0,\infty )$ defined by,

$$I_{(M,g)} (t) = \inf \{ A(\partial\Omega ) : V(\Omega) = t , \Omega \subset M^m, \ \Omega \text{ \ a \ closed \  region} \},$$

 \noindent where $V(\Omega)$, the volume of a closed region $\Omega \subset M^m$,  denotes the $m-$dimensional Riemannian measure of $\Omega$. Meanwhile,  $A(\partial \Omega)$, the area of its boundary, denotes  the $(m-1)$-dimensional Riemannian measure of the boundary $\partial \Omega$.  If a closed region $K\subset M$ realizes the infimum of the isoperimetric profile, it is called an isoperimetric region.   See, for example, \cite{Rito}, \cite{Ros}, for more details on the isoperimetric profile and  the isoperimetric   problem in general. 
 
 
The study of the isoperimetric profile and of isoperimetric regions is a classical problem in geometry. Nevertheless, the precise  isoperimetric profile is known for few manifolds. Examples include space forms: the euclidean space $\re^n$ with the flat metric, the $n$-sphere $S^n$ with the round metric and the hyperbolic space,  $\mathbb{H}^n$, with the hyperbolic metric. Also known is the isoperimetric profile of some products of  space forms with one dimensional manifolds: $S^n\times \re$ with the product metric, by the work of R. Pedrosa \cite{Pedro}; some  products of space forms with one dimensional circles $S^2\times S^1$, $\re^n\times S^1$ and $\mathbb{H}^n\times S^1$, $2\leq n\leq 7$, by the work of Pedrosa and Ritor\'e \cite{PeRit}, among others. We also mention some recent work of C. Viana  on $\re \mathbb P^n$ and space lenses \cite{Viana, Viana2}.
 
   On the other hand, the explicit isoperimetric profiles of seemingly  simple products of compact manifolds with the product metric, like $S^2\times S^2$ or $S^2\times S^1\times S^1$ are not known. In general, the isoperimetric profile of a product is not easy to describe precisely, even if one understands that of each of its factors. 
   
 Some qualitative studies in this direction, which include general lower bounds for the isoperimetric profiles of product manifolds, or characterizations of   isoperimetric regions for products,  have  then been of interest, see for example, the work of F. Morgan \cite{Morgan}, where a general lower bound for the isoperimetric profile of product manifolds is studied. See also \cite{Morgan2}, \cite{Morgan3}, \cite{Pet}, \cite{PetRu1}, \cite{PetRu2}, \cite{RuizVaz}, \cite{RuizVaz2}. 

In the work of  M. Ritor\'e and E.  Vernadakis \cite{RiVer}, it is proven that on a product of a compact manifold and the Euclidean space with the flat metric, $(N^m\times \re^n,g+g_E)$,  regions of the type $N\times B^n_r$ are isoperimetric regions for $r>r_0$, for some $r_0>0$ big enough; being $B^n_r\subset \re^n$ a  ball of radius $r$. See also the work of J. Gonzalo \cite{Gonz}. Following this lead, one may ponder if it is still the case for general products of  very small and very large volume manifolds, that the product of one factor and an isoperimetric region of the other factor is in fact an isoperimetric region of the product manifold. See for example the recent work of E. Vernadakis \cite{Verna} for results of existence in this direction. Moreover, one may ask if it  is possible to quantify how large and how small these volumes need to be, with respect to each other, for this to happen. 

In this work, we study the isoperimetric profile of products of compact manifolds $(M^m,g)$ and $(N^n,h)$. In particular, we consider a one parameter family of product manifolds with the same volume, $(X,G_{\lambda})=(M^m\times N^n,\lambda^{2n}g+ \lambda^{-2m}h)$, $\lambda>0$,  and study their isoperimetric profile for large values of $\lambda$, through some   ideas developed by F. Morgan \cite{Morgan} and A. Ros \cite{Ros}.

Our first result  estimates a lower bound for the isoperimetric profile $(M^m\times N^n,\lambda^{2n}g+ \lambda^{-2m}h)$, for large  values of $\lambda$.

\begin{Theorem}
	\label{Thm1}
  Let $(M^m,g)$, $(N^n,h)$ be compact Riemannian manifolds without boundary,  $m,n\geq 2$, concave isoperimetric profiles and volumes $V_M$, $V_N$ respectively. Let $(X,G_{\lambda})=(M^m\times N^n,\lambda^{2n}g+ \lambda^{-2m}h)$.	Let $\alpha \in (0,1)$. Then, for $v_0 \in (0, V_MV_N)$, there is some $\lambda_{0}$, such that for $\lambda>\lambda_0$ 
	
	\begin{equation}
		\label{alpha}
(2\alpha-1)^{2} f_{M,\lambda}(v_0) \leq	I_{(X,G_{\lambda})}(v_0)\leq  f_{M,\lambda}(v_0)
	\end{equation}
	
\noindent	where  $f_{M,\lambda}(v)= \lambda^{-n} V_N I_{(M,g)}(\frac{v}{V_N})$.

 Moreover,  $\lambda_0>\left(  \frac{\alpha V_N^2 }{\gamma   v_0  (1-\alpha)^2} I_{(M,g)}(\frac{v_0}{V_N})\right)^{\frac{1}{n+m}} $, where $\gamma= I_{(N,h)}(\alpha V_N)$.
\end{Theorem}


Note that $f_{M,\lambda}(v)$ gives  the area of the boundary of closed regions of volume $v$, that are the product of the factor $(N,\lambda^{-2m}h)$ and an isoperimetric region of  $(M,\lambda^{2n}g)$. The closer to 1 one chooses $\alpha$ to be, the closer the bound is to the isoperimetric profile, although $\lambda$  may need to get bigger. This is,  by increasing $\lambda$, the area of the boundary of a region of the type $ K^{\lambda} \times N_{\lambda}\subset (X,G_{\lambda}) $, gets arbitrarily close to the value of the area of the boundary of an isoperimetric region of the same  volume as  $ K^{\lambda} \times N_{\lambda}$, where $K^{\lambda}$ is an isoperimetric region of $(M,\lambda^{2n} g)$ and  $N_{\lambda}=(N, \lambda^{-2m}h)$. Moreover, we quantify an explicit lower bound for $\lambda$ in terms of $\alpha$, $v_0$, $V_N$ and the isoperimetric profiles of  the factor manifolds.  Note that the upper bound of equation (\ref{alpha}) is trivial since it is precisely the area of the boundary of a region of the type $ K^{\lambda} \times N_{\lambda}$, described above.

As our second result, we  consider the case where the first factor is the $2-$sphere with the round metric, $(M,g)=(S^2,g_0)$, and the second factor the 2-torus with the flat metric $(T^2,h)=(S^1\times S^1, dt^2+ds^2)$. In the setting $(X,G_{\lambda})=(S^2\times T^2,\lambda^{2}g_0+ \lambda^{-2}h)$, we provide an interval of volumes $(v_0, V_MV_N-v_0)$ and some $\lambda^*>0$, such that for $\lambda>\lambda^*$, regions of the type  ${ D^{\lambda}_R\times T^2_{\lambda} }$, with  $Vol({ D^{\lambda}_R\times T^2_{\lambda} })\in (v_0, V_MV_N-v_0)$,  are actual isoperimetric regions; being $D^{\lambda}_R$ a disk in $(S^2,\lambda^{2}g_0)$ and $T^2_{\lambda}=(T^2, \lambda^{-2}h)$.  Moreover, we provide an estimation of how big $\lambda^*$ must be and for which interval of volumes $v_0>0$ this holds.

 \begin{Theorem}
	\label{Thm2}
	Let $(T^2,h)=(S^1\times S^1,ds^2+dt^2)$ be the 2-torus with the flat metric and each $(S^1,dt^2)$ of radius 1. Let $(S^2,g_0)$ be the $2-$sphere with the round metric and radius 1.  Let $V$ denote the volume of $(S^2\times T^2,g_0+ h)$. Then, for any $v_0 \in (0,\frac{V}{2})$, there is some $\lambda^*>1$ such that for any $\lambda>\lambda^*$ and  any $v\in  (v_0,V-v_0)$,  ${ D^{\lambda}_R\times T_{\lambda}^2}$ is an isoperimetric region in	
	$(S^2\times T^2,\lambda^{2}g_0+ \lambda^{-2}h)$. $T^2_{\lambda}$ denotes $(T^2,\lambda^{-2}h)$ and $D^{\lambda}_R$ is a disk of radius $R>0$ in $(S^2,\lambda^{2}g_0)$,  such that $Vol({ D^{\lambda}_R\times T^2_{\lambda}})=v$.
	
	Moreover, let $\kappa_0>2$ be such that $\frac{1}{\kappa_0}V=v_0$. And $\lambda_0$ such that 
	$$\kappa_0=  \frac{\lambda_0^2}{4\pi^2}\ \  I_{S^2\times \re}\left(\frac{4\pi^2}{\lambda_0^2}\right)$$
	\noindent where $I_{S^2\times \re}$ is the isoperimetric profile of $(S^2\times \re,g_0+dt^2)$.	Then the theorem holds for $\lambda^*>\max\{\lambda_0, 2\pi, \kappa_0-1\}$.


\end{Theorem}

\noindent Note that $\lambda^*$ can be computed explicitly since the isoperimetric profile of $(S^2\times \re,g_0+dt^2)$ was computed explicitly in \cite{Pedro}. Moreover, for any (closed or homogeneous) Riemannian n-manifold $(M^n, g)$, $n\geq2$, one has $\lim_{v\rightarrow0}\frac{I(v)}{v}=\infty$ (Lemma \ref{c0}); hence $\kappa_0$ in Theorem $\ref{Thm2}$ can be made as big as necessary to fit any  $v_0\in (0,\frac{V}{2})$. Note also that Theorem \ref{Thm2} describes precisely the isoperimetric profile of $I_{(S^2\times T^2, G_{\lambda})}$ for an interval  $(v_0,V-v_0)$, for  $\lambda>\lambda^*$.


\textbf{Acknowledgments.}
The authors were supported by grant UNAM-DGAPA-PAPIIT IN108824 and grant UNAM-DGAPA-PAPIIT IN108224, respectively. The authors would like to thank Professors Detang Zhou and Xu Cheng for their hospitality at Universidade Federal Fluminense, Niteroi, Brazil,  where part of this work was done. We would also like to thank Professor Frank Morgan for his generosity in sharing some problems and conjectures related to the subject on this document.

\section{Notation and background}

Let $(M^m,g)$ and $(N^n,h)$ be closed manifolds (compact without boundary), we will denote by $(X,G_{\lambda})$  the one parameter family of products $(X,G_{\lambda})= (M^m\times N^n, \lambda^{2n} g+ \lambda^{-2m}h)$.

 Given a measurable set $E\subset X$, we denote its volume $V(E)=Vol^{n+m}(E)$ and the area of its boundary $A(\partial E)=Vol^{n+m-1}(\partial E)$. In general, we denote by $H^k(E)$ the $k-dimensional$ Hausdorff measure of a set $E$. 
 
Existence and regularity of isoperimetric regions is a fundamental result due to the works of Almgren \cite{Almgren}, Gr\"uter \cite{Gruter}, Gonzalez, Massari, Tamanini \cite{Gonzalez}, (see also Morgan \cite{MorganBook}, Ros \cite{Ros}).

\begin{Theorem}
	\label{existence}
	Let $M^m$ be a compact Riemannian manifold, or non-compact with $M/G$ compact, being $G$ the
	isometry group of $M$. Then, for any $t$, $0 < t < V(M)$, there exists a
	compact region $\Omega \subset M$, whose boundary $\Sigma = \partial \Omega$ minimizes area among regions of
	volume $t$. Moreover, except for a closed singular set of Hausdorff dimension at most
	$m - 8$, the boundary $\Sigma$ of any minimizing region is a smooth embedded hypersurface
	with constant mean curvature.
	
\end{Theorem}

The following  properties of the isoperimetric profile are well known. 

\begin{Lemma}\label{homog}{\normalfont(\cite{Bayle} p. 20)}
	
Let $(M^m,g)$ be a closed manifold of volume $V_M$, $m\geq2$. For $\lambda>0$,

$$\forall v \in [0,V_M], \  \ \ I_{(M,\lambda^2 g)}(v)=\lambda^{m-1} I_{(M,g)}\left(\frac{v}{\lambda^m}\right)$$
		
\end{Lemma}

\begin{Lemma}\label{symm}{\normalfont(Lemma 3.4 in \cite{Rito})}
	
	Let $(M^m,g)$ be a closed manifold of volume $V_M$, $m\geq2$. Then
	
	$$\forall v \in [0,V_M], \  \ \ I_{(M, g)}(v)=I_{(M,g)}(V_M-v)$$
	
\end{Lemma}

\begin{Lemma}
	\label{c0}
	Let $(M^m,g)$ be a closed manifold of volume $V_M$, $m\geq2$. For $v_0 \in (0,V_M)$, there is a constant $c_0=c_0(v_0)>0$ so that, 	for $v \in (0,v_0)$, $$I(v)\geq c_0 \ v^{\frac{m-1}{m}}.$$

\end{Lemma}


\section{Proof of Theorem \ref{Thm1}}

	We begin with a construction by F. Morgan in \cite{Morgan}, that estimates lower bounds for isoperimetric profiles of products.

For $(M^m,g)$, $(N^n,h)$ as in the hypothesis of Theorem \ref{Thm1} and $\lambda>0$, consider the product manifold  $(0,\lambda^{mn}V_M)\times (0,\lambda^{-mn}V_N)\subset \re^2$ 
	where each factor has a model metric in the sense of the Ros Product Theorem (Theorem 3.7 in \cite{Ros}), constructed as follows.   Consider $(0,\lambda^{mn} V_M)$  and $(0,\lambda^{-mn}V_N)$ with  Euclidean Lebesgue Measure and Riemannian metric $\frac{1}{\varphi_{\lambda}(x_{\lambda})}dx_{\lambda}$ and $\frac{1}{\psi_{\lambda}(y_{\lambda})}dy_{\lambda}$ respectively, where $\varphi_{\lambda}(x_{\lambda})$ and $\psi_{\lambda}(y_{\lambda})$ denote the   isoperimetric profiles of $(M,\lambda^{2n}g)$ and $(N,\lambda^{-2m}h)$ respectively and $x_{\lambda}\in   (0,\lambda^{mn} V_M)$,   $y_{\lambda}\in   (0,\lambda^{-mn} V_N)$.	We will also denote by  $\varphi(x)$ and $\psi(y)$  the   isoperimetric profiles of $(M,g)$ and $(N,h)$, respectively with  $x\in   (0, V_M)$,   $y \in   (0,  V_N)$.
	
We first prove that the given metric of each factor is in fact a model metric, in the sense of Theorem 3.7 in \cite{Ros}. Note that it suffices to prove that in each factor, $(0,\lambda^{ mn} V_M)$ and $(0, \lambda^{-mn}V_N)$, intervals of the type $(0,t)$, $t>0$, minimize perimeter, among closed sets $S$ of given Euclidean length $t$.   Let $S\subset(0,\lambda^{mn}V_M)$ be a closed set of length $t$,  and suppose $S$ is not of the type $(0,t)$; then it must be a locally finite collection of closed intervals. It follows that an interior interval must be  borderline unstable  because  the isoperimetric profile   $\varphi(x)$ is  concave. Hence $S$ does not minimize perimeter. The same argument applies for $(0, \lambda^{-mn}V_N)$, since  $\psi(y)$ is also concave.

	On the interval $(0,\lambda^{mn}V_M)$, Minkowski content   counts boundary points of intervals with density  $ \varphi_{\lambda}(x_{\lambda})=\lambda^{n(m-1)}\varphi(\frac{x_{\lambda}}{\lambda^{mn}})$, $x_{\lambda}\in (0,\lambda^{mn}V_M)$, by homogeneity of the isoperimetric profile $\varphi$ (Lemma \ref{homog}). Similarly on the  interval $(0,\lambda^{-mn}V_N)$, with density  $\psi_{\lambda}(y_{\lambda})=\lambda^{-m(n-1)}\psi(\frac{y_{\lambda}}{\lambda^{-mn}})$,  $y_{\lambda}\in (0,\lambda^{-mn}V_N)$. Minkowski content on the stripe
	$(0,\lambda^{mn}V_M)\times (0,\lambda^{-mn}V_N)$ then has perimeter measured by 
	
	\begin{equation}
		\label{Morg}
		ds^2= \psi_{\lambda}^2(y_{\lambda}) dx_{\lambda}^2+  \varphi_{\lambda}^2( x_{\lambda} ) dy_{\lambda}^2,
	\end{equation}
 	
\noindent 	for $(x_{\lambda},y_{\lambda})\in 	(0,\lambda^{mn}V_M)\times (0,\lambda^{-mn}V_N)$.

	Given $\lambda>0$, it follows from the proof of the Ros Product Theorem that, for each $v_0\in(0,V_NV_M)$, there is some closed region $E^{\lambda} \subset (0,\lambda^{mn}V_M)\times (0,\lambda^{-mn}V_N)$, with area $V(E^{\lambda})=v_0$, such that $I_{(X,G_{\lambda})}(v_0)$, is bounded from below by the perimeter $P(E^{\lambda})$ of the boundary $\delta E^{\lambda}$. $E^{\lambda}$ is symmetrized in 	$(0,\lambda^{mn}V_M)\times (0,\lambda^{-mn}V_N)$ with respect to each factor. This is, $E^{\lambda}\cap \{x_{\lambda}\}\times   (0,\lambda^{-mn}V_N)$ is an interval $\{x_{\lambda}\}\times [0, r(x_{\lambda})]$ and $E^{\lambda}\cap   (0,\lambda^{mn}V_M)\times \{y_{\lambda}\}$ is an interval $ [0, r(y_{\lambda})]\times \{y_{\lambda}\}$. In particular, $E^{\lambda}$ is such that $\delta E^{\lambda}$ is a connected boundary curve along which $y_{\lambda}$ is nonincreasing and $x_{\lambda}$ is nondecreasing, $(x_{\lambda},y_{\lambda})\in \delta E^{\lambda}\subset  (0,\lambda^{mn}V_M)\times (0,\lambda^{-mn}V_N)$. The enclosed region $E^{\lambda}$ is on the lower left of $\delta E^{\lambda}$.   This is
	
	\begin{equation}
	\label{bound}
	P(E^{\lambda})\leq I_{(X,G_{\lambda})}(v_0)
		\end{equation}
	where

	\begin{equation}
		\label{PE}
		P(E^{\lambda}) = \int_{\delta E^{\lambda}}\sqrt{ \psi_{\lambda}^2(y_{\lambda}) dx_{\lambda}^2+  \varphi_{\lambda}^2( x_{\lambda} ) dy_{\lambda}^2}
	\end{equation}
	
	$$=\int_{\delta E^{\lambda}} \sqrt{\lambda^{-2m(n-1)} \psi^2(\frac{y_{\lambda}}{\lambda^{-mn}}) dx_{\lambda}^2 + \lambda^{2n(m-1)}\varphi(\frac{x_{\lambda}}{\lambda^{mn}})^2 dy_{\lambda}^2}$$

	by homogeneity of the profiles (Lemma \ref{homog}). And the area of the region $E^{\lambda}$ is 
	\begin{equation}
	\label{Area}
	V(E^{\lambda}) = \int \int_{E^{\lambda}} dx_{\lambda} \, dy_{\lambda}.
		\end{equation}

	We define	 the transformation $J_{\lambda}:(0,\lambda^{mn}V_M)\times (0,\lambda^{-mn}V_N)\rightarrow (0,V_M)\times (0,V_N)$, for $\lambda>0$, by $J_{\lambda}(x_\lambda,y_\lambda)=(\lambda^{-mn}x_\lambda,\lambda^{mn}y_\lambda)$. We denote $(x,y)=(\lambda^{-mn}x_\lambda,\lambda^{mn}y_\lambda)$ and $\Omega_{\lambda}=J_{\lambda}(E^\lambda)\subset (0,V_M)\times(0,V_N)$. Note that $V(E^{\lambda})=V(\Omega_{\lambda})=v_0$.

	Equations (\ref{bound}) and (\ref{PE}), together with Lemma \ref{homog} imply the following lower bounds
	
	\begin{equation}
		\label{varphi}
		I_{(X,G_\lambda)}(v) \geq \int_{\delta \Omega_{\lambda}} \lambda^{n(m-1)} \varphi(\frac{\lambda^{mn}x}{\lambda^{mn}})  (\lambda^{-mn}dy) =\lambda^{-n} \int_{\delta \Omega_{\lambda}} \varphi(x)  dy
	\end{equation}
	
	and
	\begin{equation}
		\label{psi}
		I_{(X,G_\lambda)}(v)  \geq  \int_{\delta \Omega_{\lambda}} \lambda^{-m(n-1)} \psi(\frac{\lambda^{-mn}y}{\lambda^{-mn}})  (\lambda^{mn}dx) = \lambda^{m} \int_{\delta \Omega_{\lambda}}   \psi(y)   dx. 
	\end{equation}


 Also, by Lemma \ref{symm}, $I_{(X,G_{\lambda})}(v)=I_{(X,G_{\lambda})}(V_MV_N-v)$, for $v\leq \frac{V_MV_N}{2}$. Hence, in the following we will consider 
 
 \begin{equation}
 	\label{half}
 	v_0\leq  \frac{V_MV_N}{2}.
 \end{equation}
 
  Let $t_0=\frac{v_0}{V_N}$. Fix $\alpha\in (0,1)$ as in the hypothesis. Being $E^{\lambda}$ symmetrized, then $\Omega_{\lambda}$ is symmetrized. Hence, the set $\{(x,y)\in \delta \Omega_{\lambda}|x>\alpha t_0\}$ is not empty, since otherwise $V(\Omega_{\lambda})\leq \alpha t_0V_N<v_0$. Let
\begin{equation}
	\label{kappa1}
	\kappa_1 =\frac{1}{\alpha V_N} \psi(\alpha V_N). \\
		\end{equation}


 
   Since $\psi$ is concave, with $\psi(0)=\psi(V_N)=0$,   we have $\psi(y)\geq \kappa_1 y$ for $y\in (0,  \alpha V_N)$. 
   In the following, let 
 
 \begin{equation}
 	\label{lambdakappa}
 	\lambda> \left(\frac{1}{  \kappa_1   (1-\alpha)^2 }\frac{\varphi(t_0)}{t_0} \right)^{\frac{1}{n+m}}=\left(  \frac{\alpha V_N^2 }{\gamma   v_0  (1-\alpha)^2} I_{(M,g)}\left(\frac{v_0}{V_N}\right)\right)^{\frac{1}{n+m}} 
\end{equation}

\noindent  where $\gamma= I_{(N,h)}(\alpha V_N)$.
  
  We note that the upper bound of the Theorem is trivial, since it is the area of the boundary of a closed region
  $K^{\lambda}\times N_{\lambda} \subset (X,G_{\lambda})$, where $N_{\lambda}=(N,\lambda^{-2m}h)$ and $K^{\lambda}$ is an isoperimetric region of $(M^m,\lambda^{2n}g)$ with volume $V(K^{\lambda})=\lambda^{mn}\frac{v_0}{V_N}$. 	 By direct computation and homogeneity of the isoperimetric profile (Lemma \ref{homog}),  
  
  \begin{equation}
  	\label{uppergral}
  	I_{(X,G_{\lambda})}(v_0)\leq	A(\partial(K^{\lambda}\times N_{\lambda}))=\lambda^{-n} \  \ V_N \  \ I_{(M^m,g)}\left(\frac{v_0}{V_N}\right)
  \end{equation}

    We now proceed by cases to prove Theorem \ref{Thm1}.
 \begin{enumerate}
 	\item \textbf{Case 1.} \textit{For $(x,y)\in \delta \Omega_{\lambda}$, if $x>\alpha t_0 $, then $y\leq \alpha V_N $}
 	
 	Let 
 	$$R=\{(x,y)\in (0, V_M)\times (0, V_N)| \alpha t_0\leq  x\leq  t_0   \  \textit{and} \    y\geq \alpha V_N \}$$
 	
	$$W=\{(x,y)\in (0, V_M)\times (0, V_N)| 0\leq x\leq t_0 , 0\leq y \leq V_N \}$$
 	
 	Note that $V(R)=(1-\alpha)^2t_0V_N$. Also $R\subset W\setminus \Omega_{\lambda}$ and $V(W\setminus \Omega_{\lambda})=V( \Omega_{\lambda}\setminus W)$. Hence	
 	$V( \Omega_{\lambda}\setminus W) \geq V(R)=(1-\alpha)^2t_0V_N   $.


Using $\kappa_1$ as in eq. (\ref{kappa1}), together with eq. (\ref{psi})  yields

 $$I_{(X,G_\lambda)}(v)\geq \lambda^m\int_{\delta  \Omega_{\lambda}}\psi(y) dx \geq \lambda^m\int_{\delta   \Omega_{\lambda}\setminus W} \psi(y) dx \geq \lambda^m\int_{\delta  \Omega_{\lambda}\setminus W} \kappa_1 y dx$$
 $$ \geq\lambda^m \kappa_1 \ \ \int_{\delta  \Omega_{\lambda}\setminus W} y \ dx= \lambda^m \kappa_1\  V( \Omega_{\lambda}\setminus W) \geq \lambda^m \kappa_1 \ \ (1-\alpha)^2 \ t_0 \ V_N$$
 
 since $\int_{\delta  \Omega_{\lambda}\setminus W} y dx=V( \Omega_{\lambda}\setminus W)$.
 
Together with (\ref{uppergral}), we get

 $$\lambda^{-n} V_N \varphi(t_0)\geq I_{(X,G_\lambda)}(v_0)\geq \lambda^m \kappa_1 \ \ (1-\alpha)^2 \ t_0 \ V_N$$

 this is,
 $$\lambda^{m+n} \leq  \frac{\varphi(t_0)}{t_0 \ \kappa_1 \ \ (1-\alpha)^2 }   $$
 
 	We conclude that this case cannot be possible for $\lambda$   as in eq. (\ref{lambdakappa}).

 	 	\item \textbf{Case 2.} \textit{There is some $(x_1,y_1)\in \delta \Omega_{\lambda}$, such that $x_1>\alpha t_0$ and $y_1> \alpha V_N $.}%
 	
 	We subdivide this in cases.
 	
 	\begin{enumerate}
 		\item \textbf{Case 2.a} \textit{For all $(x,y)\in \delta \Omega_{\lambda}$,  $x\leq  (V_M-(2\alpha-1) t_0)$.
 			Or, for all $(x,y)\in \delta \Omega_{\lambda}$, if $x>V_M-(2\alpha-1) t_0 $ then  $y\leq (1-\alpha) V_N$.}

 Since $\varphi$ is concave and symmetric,
 	 	$\varphi((2\alpha-1) t_0)\leq \varphi(x)$ for $x$ such that $(2\alpha-1) t_0\leq x\leq V_M-(2\alpha-1) t_0$. Hence, by equation (\ref{varphi})
 	
 	 \begin{equation}\label{Omega1}
 	 	I_{(X,G_{\lambda})}(v_0)\geq \int_{\delta\Omega_{\lambda}} \lambda^{-n}\varphi(x)dy \geq \int_{\delta\Omega_{\lambda}\cap U} \lambda^{-n}\varphi(x)dy \geq \lambda^{-n} \varphi((2\alpha-1) t_0) \int_{\delta\Omega_{\lambda}\cap U} dy
 	 \end{equation} 
 	
 	\noindent where $$U=\{(x,y)\in (0,V_M)\times (0,V_N)| (2\alpha-1) t_0 \leq x\leq V_M-(2\alpha-1) t_0\}$$
 	\noindent Now, using the hypothesis of this case, we have some $y_1> \alpha V_N$, $(x_1,y_1)\in \delta \Omega_{\lambda}\cap U$, and some  $y_2\leq (1-\alpha) V_N$, $(x_2,y_2)\in \delta \Omega_{\lambda}\cap U$. Also, since  $\Omega_{\lambda}$ is symmetrized, $x_1\leq x_2$. This yields
 	
\begin{equation}\label{Omega2}
	 \lambda^{-n} \varphi((2\alpha-1) t_0)  \int_{\delta\Omega_{\lambda}\cap U} dy	\geq \lambda^{-n}  \varphi((2\alpha-1) t_0) (\alpha V_N-(1-\alpha)V_N)
	  	 \end{equation}

 \noindent	Moreover, since $\varphi$ is concave and $\varphi(0)=0$, then $\varphi((2\alpha-1) t_0)\geq (2\alpha-1) \varphi(t_0)$, for $0< \alpha< 1$. Together with (\ref{Omega1}) and (\ref{Omega2}), this implies 
 $$I_{(X,G_{\lambda})}(v_0)\geq \lambda^{-n} \varphi( t_0) (2\alpha-1)^2  V_N$$

 If we recall that $t_0= \frac{v_0}{V_N}$, we have the lower bound  of Theorem \ref{Thm1}.

 
  		\item \textbf{Case 2.b} \textit{For all $(x,y)\in \delta \Omega_{\lambda}$, if $x> V_M-(2\alpha-1) t_0 $ then  $y\leq \alpha V_N $. Also, there is some $(x_0,y_0)$ such that $x_0>V_M-(2\alpha-1) t_0$ and $y_0>(1-\alpha)V_N$.}

  		We consider two subcases: 
  		
  	{ \textbf{i)} } { \textbf{Case 2.b, i)} }  {\textit{ For $(x,y)\in \delta \Omega_{\lambda}$, if $x> V_M-\alpha t_0$ then $y\leq \alpha V_N$}}.
 
  	Let 
 
 $$S=\{(x,y)\in (0, V_M)\times (0, V_N)|  (V_M-\alpha t_0)\leq x\leq V_M-(2\alpha-1) t_0,\ 0\leq y\leq \alpha V_N \}$$
  $$T=\{(x,y)\in (0, V_M)\times (0, V_N)|  (V_M-\alpha t_0)\leq x\leq V_M-(2\alpha-1) t_0,\ 0\leq y\leq (1-\alpha) V_N \}$$
 
 Note that $T\subset S$ and $V(T)=(1-\alpha)^2 t_0 V_N$. Also, in this case, $T\subset \Omega_{\lambda}$. 
 
  By eq. (\ref{psi}) 
  
  $$I_{(X,G_\lambda)}(v_0)\geq \lambda^{m}\int_{\delta  \Omega_{\lambda}} \psi(y)dx\geq \lambda^{m}\int_{\delta  \Omega_{\lambda}\cap S} \psi(y)dx  $$
  $$\geq \lambda^m\int_{\delta \Omega_{\lambda}\cap S} \kappa_1 y dx\geq\lambda^m \kappa_1 \ V(T)= \lambda^m \kappa_1 \ \ (1-\alpha)^2 \ t_0 \ V_N$$
  
  where we have used $\kappa_1$ as in  (\ref{kappa1}) so that $\psi(y)\geq \kappa_1 y$, for $y \in [0, \alpha V_N]$; which includes all $(x,y) \in \delta  \Omega_{\lambda}\cap S$.   
  From this lower bound and the upper bound  for $I_{(X,G_\lambda)}(v_0)$, equation (\ref{uppergral}), it follows
  
  $$\lambda^{m+n} \leq  \frac{\varphi(t_0)}{t_0\   \kappa_1 \ \ (1-\alpha)^2 }   $$

 which is not possible by (\ref{lambdakappa}).
 
{ \textbf{	ii)} } { \textbf{Case 2.b, ii)} }   {\textit{ There is some $(x_3,y_3)\in \delta \Omega_{\lambda}$, such that  $x_3\geq V_M-\alpha t_0$ and $y_3>\alpha V_N$.}}
 
 Let $\epsilon=y_3-\alpha V_N>0$, and consider the following regions in $(0, V_M)\times (0, V_N)$:
 
  $$A_1=\{(x,y)| \ 0\leq  x\leq  V_M- \alpha  t_0,\ 0\leq y\leq  \alpha  V_N \}$$
 $$A_2=\{(x,y)|\ V_M- \alpha t_0\leq x\leq V_M-(2\alpha-1) t_0,\ 0\leq y\leq (1-\alpha) V_N \}$$
  $$A_3=\{(x,y)| \  0\leq x\leq V_M-\alpha t_0,\   \alpha  V_N \leq y\leq  \alpha  V_N+\epsilon \}$$
 
 Since $\Omega_{\lambda}$ is symmetrized $A_1\cup A_2\cup A_3\subset \Omega_{\lambda}$. Also $V(A_i\cap A_j)=0$, for $i\neq j$, $i,j=1,2,3$. Hence
 
\begin{equation}
	\label{sum1}
	V(A_1)+V(A_2)+V(A_3)\leq V(\Omega_{\lambda})=v_0
\end{equation}
 Direct computation yields
 
 $V(A_1)=(V_M-\alpha t_0)\alpha V_N$, $V(A_2)=(1-\alpha)^2v_0$ and $V(A_3)=(V_M-\alpha t_0)\epsilon$. Then, eq. (\ref{sum1}) implies
 
 $$v_0=V(\Omega_{\lambda})\geq  \alpha V_M V_N+(1-2\alpha)v_0+\delta$$
 
 \noindent where $\delta=(V_M-\alpha t_0)\epsilon>0$. It follows
 
 $$  v_0 \geq \frac{   V_M V_N}{2}+\frac{\delta}{2\alpha}>\frac{   V_M V_N}{2}$$
 
 This contradicts   $v_0\leq \frac{V_M V_N}{2}$.
 
 	\end{enumerate}

 \item {\textbf{Case 2.c}} \textit{There is some $(x_0,y_0)\in \delta \Omega_{\lambda}$, such that $x_0\geq V_M-(2\alpha-1) t_0$ and  $y_0> \alpha V_N $.}
 Let 	
 $$B=\{(x,y)\in (0, V_M)\times (0, V_N)| 0\leq x\leq V_M-(2\alpha-1) t_0,\ 0\leq y\leq \alpha V_N \}$$
 
 By hypothesis and since $\Omega_{\lambda}$ is symmetrized, $B\subset \Omega_{\lambda}$. Hence
$ V(B)\leq V(\Omega_{\lambda})=v_0$.
Direct computation then yields

$$ (V_M-(2\alpha-1)t_0)\alpha  V_N=V(B) \leq V(\Omega_{\lambda})=v_0$$
\noindent which implies,
$$\frac{V_MV_N}{2}\leq \frac{2\alpha^2-\alpha+1}{2\alpha} v_0<v_0$$

\noindent since $\frac{2\alpha^2-\alpha+1}{2\alpha}<1$ for $\alpha\in(0,1)$. This contradicts  $v_0\leq \frac{V_M V_N}{2}$.
 	
 \end{enumerate}
 

 \section{Proof of Theorem \ref{Thm2}}

 
 The isoperimetric profile of $(S^2\times S^1,g_0+dt^2)$, with $g_0$ the round metric, was computed by Pedrosa and Ritor\'e in \cite{PeRit}. It has been conjectured that  $(S^n\times S^1,g_n+dt^2)$ for $n>2$ exhibits a similar behavior, except maybe, with the addition of regions enclosed by hypersurfaces generated by unduloids, although, to our knowledge, this is not known yet. If true, the techniques used here could be used to prove similar results to those of Theorem \ref{Thm2} on $ S^n \times T^2$ for $n>2$.
 We begin by recalling the following.

 \begin{Theorem} \label{Perit}\textnormal{(Pedrosa and Ritor\'e, \cite{PeRit})}
 	The isoperimetric domains in the Riemannian product $(S^1(r)\times S^2)$ are either
 	\begin{enumerate}
 		\item balls or complements of balls,
 		\item tubular neighborhoods of the closed geodesics $S^1\times \{point\}$, which are diffeomorphic to $S^1\times S^1$,
 		\item sections bounded by two totally geodesic $\{point\}\times S^2$, which are diffeomorphic to $[a,b]\times S^2$.
 		
 		If $r$ is such that $r<\frac{\omega_2}{\pi \omega_1}$ then the sections are not solutions. 
 	\end{enumerate}
 \end{Theorem}
 
Theorem \ref{Perit} implies
 
 $$I_{(S^2\times S^1, g_0 + dt^2)}(w)=min\{\psi_1(w), \psi_2(w),\psi_3(w)\}$$
 
 \noindent  where $\psi_1(w)$ gives the area of geodesic balls in $S^2\times \re$ of volume $w$, for $w\in[0,4\pi^2]$, and $\psi_1(w):=\psi_1(8\pi^2-w)$, for $w\in (4\pi^2,8\pi^2]$, to take the complements of balls into account;  $\psi_2(w)$ gives the area of regions of the type $S^1\times D_R$, where $D_R$ is a disk of $S^2$ with radius $R>0$ such that $S^1\times D_R$  has volume $w$; and $\psi_3(w)$ gives the area of sections of volume $w$, described in (3) of Theorem \ref{Perit}. Note that $\psi_3(w)$ is constant.


The areas and volumes   of geodesic balls in $S^2\times \re$, were computed by R. Pedrosa in \cite{Pedro} to be, in parametric form,
 \begin{equation}\label{Ah}
 	A(\Omega_h)= 4 \pi  \left(\frac{2}{1 + h^2} + \frac{
 h^2}{(1 + h^2)^{3/2}} log\frac{(\sqrt{1 + h^2} + 1)}{(\sqrt{1 + h^2} - 
 1)} \right)
 \end{equation} 
 $$Vol(\Omega_h)= 4 \pi h  \left(\frac{2+h^2}{(1 + h^2)^{3/2} } log\frac{(\sqrt{1 + h^2} + 1)}{(\sqrt{1 + h^2} - 
 	1)} - \frac{
 	2}{(1 + h^2)} \right)$$
 
 \noindent where $h$ is the mean curvature of the geodesic balls, $h>0$. This system gives $\psi_1(w)$, with $\psi_1(w)=A(\Omega_h)$ and $w=Vol(\Omega_h)$. Note also that $\psi_1(w)$ is concave since $S^2\times \re$ has non-negative Ricci curvature (cf. in \cite{Bayle}).

Recall that the isoperimetric profile of $(S^2,g_0)$	is given by $I_{(S^2,g_0)}(y)=\sqrt{y(4\pi-y)}$, $y\in[0,4\pi]$. So that $\psi_2(w)= 2\pi \sqrt{y(4\pi-y)}$, $w=2\pi y$, $y\in[0,4\pi]$. 

\noindent Also, direct computation gives $\psi_3(w)=8\pi$, for any $w\in (0, 8\pi^2)=(0,Vol(S^2\times S^1))$.

We now consider the rescaling  	$(S^2\times S^1)^{\lambda}=(S^2\times S^1, \lambda^{-2} dt^2+\lambda^2 g_0)$. Let $W^{\lambda}=Vol(S^2\times S^1)^{\lambda}= \lambda 8\pi^2$.
	
The corresponding rescalings $\psi_i^{\lambda}$, for the functions $\psi_i$, $i=1,2,3$, are the following.
\begin{enumerate}[i)]
	\item $\psi_1^{\lambda}(w^{\lambda})= \lambda^2 \psi_1(x)$, with
$w^{\lambda}=\lambda^3 x$.

\noindent By homogeneity of the isoperimetric profile (Lemma \ref{homog}) and considering that $(\re\times S^2, dt^2+\lambda^2g_0)$ is isometric to $(\re\times S^2, \lambda^2(dt^2+g_0))$.

\item $\psi_2^{\lambda}(w^{\lambda})= 2\pi   \sqrt{y(4\pi-y)}$, with
$w^{\lambda}=2\pi \lambda y$, $y\in[0,4\pi]$.

	\noindent 	By homogeneity of the isoperimetric profile (Lemma \ref{homog}). 
	
\item 	$\psi_3^{\lambda}(w^{\lambda})= 8 \pi \lambda$,  for any $w^{\lambda}\in (0, 8\pi^2\lambda )$.
\end{enumerate}

	\begin{Lemma}
		\label{beta}
		 Let $\lambda_0\geq 2$. For any $\lambda\geq \lambda_0$, there is some $\beta^{\lambda}\in (0,\frac{W^{\lambda}}{2}) $, such that 
		 $$I_{(S^2\times S^1)^{\lambda}}(w^{\lambda})= \psi_2^{\lambda}(w^{\lambda})$$
		 
		 for $w^\lambda\in (\beta^\lambda, W^{\lambda}-\beta^{\lambda})$.
		 Moreover, an estimation for $\beta^\lambda$ is  $$\beta^\lambda\leq x_1^{\lambda} \lambda^3,$$ where
		  $x_1^{\lambda}$ is such that $\psi_1(x_1^{\lambda})=\frac{4\pi^2}{\lambda^2}.$

		 Another estimation  is
		 $$\beta^{\lambda}\leq\frac{ 4\pi^2}{\kappa_0} \lambda$$
		  \noindent with $\kappa_0=\psi_1(\frac{4\pi^2}{\lambda_0^2})\frac{\lambda_0^2}{4\pi^2}$.
	\end{Lemma}
	
		Note that $x_1^{\lambda}$ and $\kappa_0$ are readily computable since $\psi_1(x)$ is precisely known (equations (\ref{Ah})).
	
	\begin{proof}
		
		Note that the maximum value of $\psi_2^{\lambda}(w^{\lambda})$ is $4\pi^2$, this implies that 
		$\psi_3^{\lambda}(w^{\lambda})>\psi_2^{\lambda}(w^{\lambda})$ for $\lambda>\frac{\pi}{2}$. Hence, for $\lambda\geq2$, we may discard 	$\psi_3^{\lambda}$ as a contender for isoperimetric profile values of $I_{(S^2\times S^1)^{\lambda}}$. 
		
		Moreover, it will suffice to prove that 
		\begin{equation}\label{2bigger1}
			\psi_2^{\lambda}(w^{\lambda})\leq \psi_1^{\lambda}(w^{\lambda}),
		\end{equation}
		
	\noindent 	for any $w^{\lambda}\in (\beta^\lambda, \frac{W^{\lambda}}{2}]=(\beta^\lambda, \lambda 4\pi^2]$ (by symmetry of the isoperimetric profile, Lemma \ref{symm}).
		
		For $\psi_2^{\lambda}(w^{\lambda})\leq \psi_1^{\lambda}(w^{\lambda})$, to occur, considering the rescalings of $\psi_1$ and $\psi_2$, we must get
				$ 2\pi   \sqrt{y(4\pi-y)}\leq \lambda^2 \psi_1(x)$ and $\lambda^3 x= 2\pi  y \lambda$,
		which is equivalent to 
		\begin{equation}\label{21L}
			 \psi_1(x)\geq \sqrt{x\left(\frac{8\pi^2}{\lambda^2}-x\right)}
			\end{equation}

		Note that, for each $\lambda$, the right hand side, $h^{\lambda}(x)=\sqrt{x\left(\frac{8\pi^2}{\lambda^2}-x\right)}$, has a maximum of 
		$h^{\lambda}(x_M^{\lambda})=\frac{4\pi^2}{\lambda^2}$ at $x_M^{\lambda}=\frac{4\pi^2}{\lambda^2}$. Concavity of $\psi_1$ and direct computation show that    $\psi_1(x)> x$, $x\in (0,\pi^2]$. Hence, for each $\lambda\geq 2$, $\psi_1(\frac{4\pi^2}{\lambda^2})>\frac{4\pi^2}{\lambda^2}=\max_x\{h^{\lambda}(x)\}$.

By concavity of $h^{\lambda}(x)$, and since $ \psi_1(x)$ is increasing, then (\ref{21L}) is satisfied at least for $x\in [x_1^{\lambda}, \frac{4\pi^2}{\lambda^2}]$, where $x_1^{\lambda}$ is such that $\psi_1(x_1^{\lambda})=\frac{4\pi^2}{\lambda^2}$. This is equivalent to $	\psi_2^{\lambda}(w^{\lambda})\leq \psi_1^{\lambda}(w^{\lambda}),$ being satisfied for at least $w^{\lambda}\in [\lambda^3 x_1^{\lambda}, 4\pi^2  \lambda]$. This implies the first part of the lemma, for example, with $\beta^{\lambda}=\lambda^3 x_1^{\lambda}$.

		For a simpler, more uniform, approach to estimate $\beta^{\lambda}$, consider the line $l(x)=\kappa_0 x$ that passes through the origin and the point $(\frac{4\pi^2}{\lambda_0^2},\psi_1(\frac{4\pi^2}{\lambda_0^2}))$. Since $\psi_1$ is concave and $\psi_1(0)=0$, then $\frac{\psi_1(x)}{x}$ is decreasing. It also implies $\psi_1(x)> \kappa_0 x$, for $x\in (0,\frac{4 \pi^2 }{\lambda_0^2})$.  Hence $\kappa_0=\psi_1(\frac{4\pi^2}{\lambda_0^2})\frac{\lambda_0^2}{4\pi^2}$ is non-decreasing as $\lambda_0$ increases.  By direct computation $\frac{\psi_1(\pi^2)}{\pi^2}\approx1.96$. It follows that $\lambda_0\geq2$ implies  $\kappa_0\geq 1.96$.  Then, for any $\lambda_0\geq2$,
		 $\psi_1(x)> \kappa_0 x>x$, for $x\in (0,\frac{4 \pi^2 }{\lambda_0^2})$.
		 Hence, for each $\lambda\geq \lambda_0$, $\psi_1(\frac{4\pi^2}{\lambda^2})>\kappa_0 \frac{4\pi^2}{\lambda^2}>\frac{4\pi^2}{\lambda^2}=\max_x\{h^{\lambda}(x)\}$.
		 
		 It follows that for $x_2^{\lambda}$ such that $\kappa_0 x_2^{\lambda}=\frac{4\pi^2}{\lambda^2}$, we get inequality
		 (\ref{21L}) for at least, $x\in (x_2^{\lambda}, \frac{4\pi^2}{\lambda^2}]$. This is,
		 
		 $$	\psi_2^{\lambda}(w^{\lambda})\leq \psi_1^{\lambda}(w^{\lambda}),$$

		\noindent at least for any $w^{\lambda}\in (\lambda^3 x_2^{\lambda}, \frac{W^{\lambda}}{2}]$. This implies
	$\beta^{\lambda}\leq \lambda^3 x_2^{\lambda}=	\lambda \frac{4\pi^2}{\kappa_0}<\frac{W^{\lambda}}{2}$, for $\lambda\geq\lambda_0 $.

	\end{proof}

		Let $V=Vol( S^2\times S^1\times S^1, g_0+dt^2+ds^2)=16\pi^3$ and $G_{\lambda}= (\lambda^2g_0+\lambda^{-2}dt^2+\lambda^{-2}ds^2)$. Given $\lambda>2$, $v\in(0,V)$, let $\Omega^{\lambda}$ be an isoperimetric region in $(S^2\times S^1\times S^1,G_{\lambda})$ with volume $v\in (0,V)$. Parametrize $S^1_{\lambda}$ by  $[0,2\pi{\lambda^{-1}}]$, identifying $0$ and $2\pi{\lambda^{-1}}$, and consider a function $F_{\Omega^\lambda}: [0,2\pi{\lambda^{-1}}]\rightarrow [0,\lambda W]$, given by 
 
 \begin{equation}
 	\label{FLam}
 F_{\Omega^{\lambda}}(s)= Vol(\Omega_s^{\lambda}) 	
 \end{equation}
 with $$\Omega_s^{\lambda}=\Omega^{\lambda}\cap \left((S^2\times S^1)^{\lambda}\times \{s\} \right)$$ \noindent $s\in [0,2\pi{\lambda^{-1}}]$. We establish a lower bound for $I_{(S^2\times S^1\times S^1, G_{\lambda})}(v)$ in terms of the minimum and maximum values of $ F_{\Omega^{\lambda}}$.

\begin{Lemma}
	\label{gamma}
Let $\Omega^{\lambda}$, $\lambda\geq2$, be an isoperimetric region in $(S^2\times S^1\times S^1,G_{\lambda})$ with volume $v\in (0,V)$.		Then 
	\begin{equation}
	\label{psi2}
 I_{(S^2\times S^1\times S^1, G_{\lambda})}(v)\geq 2(F^{\lambda}_M-F^{\lambda}_m)
\end{equation}

\noindent with $F^{\lambda}_M$ and $F^{\lambda}_m$ the maximum and minimum values of the function $ F_{\Omega^{\lambda}}$ from (\ref{FLam}).

\end{Lemma}

\begin{proof}

 Let $W=Vol(S^2\times S^1, g_0 + dt^2)=8\pi^2$. Consider $J^{\lambda}=(0,\lambda W)$, $\lambda\geq 2$,   with  Euclidean Lebesgue Measure and Riemannian metric $\frac{1}{\varphi_{\lambda}(x_{\lambda})}dx_{\lambda}$, where $\varphi_{\lambda}(x_{\lambda})$  denotes the   isoperimetric profile of  $(S^2\times S^1)^{\lambda}=(S^2\times S^1,\lambda^{2}g_0+\lambda^{-2}dt^2)$  and $x_{\lambda}\in   (0,\lambda W)$.
	
Note that the given metric of $J^{\lambda}$ is a model metric, in the sense of Theorem 3.7 in \cite{Ros}. It suffices to prove that in each interval, $(0,\lambda W)$, intervals of the type $(0,r)$, $r>0$, minimize perimeter, among closed sets $S$ of given Euclidean length $r$.   Let $S\subset(0,\lambda W)$ be a closed set of length $r$,  and suppose $S$ is not of the type $(0,r)$; then it must be a locally finite collection of closed intervals. It follows that an interior interval must be  borderline unstable  because  the isoperimetric profile   $\varphi_{\lambda}(x_{\lambda})$ is  concave. Hence $S$ does not minimize perimeter. 

	On the interval $(0,\lambda W)$, Minkowski content   counts boundary points of intervals with density  $\varphi_{\lambda}(x_{\lambda})$, $x_{\lambda}\in (0,\lambda W)$.
	Minkowski content on the stripe
	$(0,\lambda W)\times S^1_{\lambda}$ then has perimeter measured by 
	
	\begin{equation}
		\label{Morg}
		ds^2=  4 dx_{\lambda}^2+  \varphi_{\lambda}^2( x_{\lambda} ) dy_{\lambda}^2,
	\end{equation}
 	
\noindent 	for $(x_{\lambda},y_{\lambda})\in 	(0,\lambda W)\times S^1_{\lambda}$, where $S^1_{\lambda}=(S^1,\lambda^{-2}ds^2)$.

 We symmetrize $\Omega^{\lambda}\subset (S^2\times S^1)^{\lambda}$ with respect to the model metric $J^{\lambda}$. We parametrize $S^1_{\lambda}$ by $[0,2\pi \lambda^{-1}]$, identifying $0$ and $2\pi \lambda^{-1}$. Let 
	
	$\Omega_s^{\lambda}=\Omega^{\lambda}\cap \left( (S^2\times S^1)^{\lambda}\times \{s\}\right)$, $s\in [0,2\pi{\lambda^{-1}}]$.	
	
We conform a region	$E^{\lambda} \subset	(0,\lambda W)\times S^1_{\lambda}$ by placing an interval of volume $Vol(\Omega_s^{\lambda})$, $[0, Vol(\Omega_s^{\lambda}) ]$, at each $s\in [0,2\pi \lambda^{-1}]$:

$$E_s^{\lambda}=E^{\lambda}\cap \left( W^{\lambda}\times \{s\}\right)=[0, Vol(\Omega_s^{\lambda}) ]\times  \{s\}=[0, F_{\Omega^{\lambda}}(s) ]\times  \{s\},$$ 

\noindent $s\in [0,2\pi{\lambda^{-1}}]$.	 It follows from the proof of the Ros Product Theorem that   $E^{\lambda} \subset	(0,\lambda W)\times S^1_{\lambda}$, has area $V(E^{\lambda})=v$, and is such that $I_{( S^2\times S^1\times S^1,G_{\lambda})}(v)$, is bounded from below by the perimeter $P(E^{\lambda})$ of the boundary $\delta E^{\lambda}$. By construction, $E^{\lambda}$ is symmetrized in 	$	(0,\lambda W)\times S^1_{\lambda}$ with respect to $W^{\lambda}$. This is $E^{\lambda}\cap   ((0,\lambda W)\times \{y_{\lambda}\})$ is an interval $ [0, r(y_{\lambda})]\times \{y_{\lambda}\}$, $y_{\lambda}\in S^1_{\lambda}$. In particular, $E^{\lambda}$ is such that $\delta E^{\lambda}$ is a connected boundary curve along which $y_{\lambda}$ is nonincreasing, $(x_{\lambda},y_{\lambda})\in \delta E^{\lambda}\subset  (0,\lambda W)\times S^1_{\lambda}$. The enclosed region $E^{\lambda}$ is on  left of $\delta E^{\lambda}$.   This is
	
	\begin{equation}
	\label{bound2}
	P(E^{\lambda})\leq I_{((S^2\times S^1\times S^1),G_{\lambda})}(v)
		\end{equation}
	where

	\begin{equation}
		\label{PE2}
		P(E^{\lambda}) = \int_{\delta E^{\lambda}}\sqrt{ 4 dx_{\lambda}^2+  \varphi_{\lambda}^2( x_{\lambda} ) dy_{\lambda}^2}
	\end{equation}
	


\noindent And the area of the region $E^{\lambda}$ is $
	V(E^{\lambda}) = \int \int_{E^{\lambda}} dx_{\lambda} \, dy_{\lambda}$.

	Equations (\ref{bound2}) and (\ref{PE2}),  imply the  lower bound

	\begin{equation}
		\label{psi3}
		I_{(S^2\times S^1 \times S^1,G_\lambda)}(v)  \geq  2 \int_{\delta E^{\lambda}} dx_{\lambda} \geq 2(F^{\lambda}_M-F^{\lambda}_m), 
	\end{equation}
	
\noindent	with $F^{\lambda}_M$ and $F^{\lambda}_m$ the maximum and minimum values of the function $ F_{\Omega^{\lambda}}$.
\end{proof}
	
For an upper bound of 	$I_{(S^2\times S^1 \times S^1,G_\lambda)}$, let $f_{\lambda}:[0,16\pi^3]\rightarrow\re$ be the function that gives the area $f_{\lambda}(v)$ of the boundary of a region of the type  $ D_R^{\lambda}\times S_{\lambda}^1\times S_{\lambda}^1$ where $D_R^{\lambda}\subset (S^2,\lambda^2g_0)$ is a disk,  such that  $Vol(D_R^{\lambda}\times S_{\lambda}^1\times S_{\lambda}^1)=v$. Note that $f_{\lambda}(v)= \frac{1}{\lambda}\sqrt{v(16\pi^3-v)}$ and then

	\begin{equation}
		\label{upper}
		I_{(S^2\times S^1 \times S^1,G_\lambda)}(v)\leq \frac{1}{\lambda}\sqrt{v(16\pi^3-v)}\leq \frac{8\pi^3}{\lambda}
		\end{equation}
	
\noindent	since $D_R^{\lambda}\times S_{\lambda}^1\times S_{\lambda}^1$, are realizable regions in $(S^2\times S^1 \times S^1,G_\lambda)$. And  $ \frac{8\pi^3}{\lambda}$ is the maximum value of $\frac{1}{\lambda}\sqrt{v(16\pi^3-v)}$, $v\in (0, 16\pi^3)$.
\begin{Lemma}
	\label{ineq}
Let $v_0\in(0,\frac{V}{2})$ and  $v\in (v_0,V-v_0)$. 
	Let $\kappa_0>2$ be such that $v_0=	V\frac{1}{\kappa_0}$. Let   $\lambda_0>0$  be such that
	
	$$\kappa_0=  \frac{\lambda_0^2}{4\pi^2}\ \  \psi_1\left(\frac{4\pi^2}{\lambda_0^2}\right).$$
	
	Let $\lambda^*>\max\{\lambda_0, 2\pi, \kappa_0-1\}$ and $\lambda\geq \lambda^*$. Let $\Omega^{\lambda}$ be  an isoperimetric region in $(S^2\times S^1 \times S^1,G_\lambda)$, with volume $v$.

Parametrize $S^1_{\lambda}$ by  $[0,2\pi{\lambda^{-1}}]$, identifying $0$ and $2\pi{\lambda^{-1}}$ and consider a function $F_{\Omega^{\lambda}}: [0,2\pi{\lambda^{-1}}]\rightarrow [0,\lambda W]$, given by

\begin{equation}
\label{Flambda}
F_{\Omega^{\lambda}}(s)= Vol(\Omega_s^{\lambda})
\end{equation}
with
$\Omega_s^{\lambda}=\Omega^{\lambda}\cap \left((S^2\times S^1)^{\lambda}\times \{s\}\right)$, $s\in [0,2\pi{\lambda^{-1}}]$.

It holds,
	$$\beta_0^{\lambda}\leq F_m^{\lambda} \leq F_M^{\lambda}\leq  W^{\lambda}-\beta_0^{\lambda}.$$
\noindent where $F_m^{\lambda}$ and $F_M^{\lambda}$ denote the minimum and maximum values of $F_{\Omega^{\lambda}}$ and  $\beta_0^{\lambda}= \lambda \frac{4\pi^2}{\kappa_0}$.
\end{Lemma}

\begin{Remark}
\label{IS}	
It is stated in \cite{Pedro} that $I_{S^\times \re}(w)$, for $w\leq w^*$, $w^*\approx 16.66$, is realized by geodesic balls. 	Since $I_{(S^2\times \re)}(v)$ is concave, and using direct computation,  $\kappa_0>2$ implies $\lambda_0>2.05$ (i.e. $\frac{4\pi^2}{\lambda_0^2}<\pi^2\approx9.87<16.66$). This is  $I_{S^2\times \re}\left(\frac{4\pi^2}{\lambda_0^2}\right)=\psi_1\left(\frac{4\pi^2}{\lambda_0^2}\right)$, for $\kappa_0>2$. This implies that the definition of $k_0$  in the  hypothesis of Lemma \ref{ineq} and of Theorem \ref{Thm2} coincide. 
\end{Remark}

\begin{proof}

Note first that the restriction on volume $v$, can be rewritten as,

	\begin{equation}
		\label{Vrestriction0}
		v_0=	V\frac{1}{\kappa_0}\leq v\leq V\left(\frac{\kappa_0-1}{\kappa_0}\right)=V-v_0
	\end{equation}

	Also,  since $\lambda>\kappa_0-1$, (\ref{Vrestriction0})  implies,
	\begin{equation}
		\label{Vrestriction1}
		\frac{V}{2}\left(\frac{1}{\kappa_0}+\frac{\kappa_0-1}{\kappa_0 \lambda}\right)\leq v\leq \frac{V}{2}\left(2-\frac{1}{\kappa_0}-\frac{\kappa_0-1}{\kappa_0 \lambda}\right)
	\end{equation}
	
	\noindent which in turn is equivalent to
	\begin{equation}
		\label{Vrestriction}
		2\pi \lambda^{-1} \left( \beta_0^{\lambda}+c^{\lambda}\right)\leq v\leq 2\pi \lambda^{-1}\left( W^{\lambda}-\beta_0^{\lambda}-c^{\lambda}\right).
	\end{equation}
	
	\noindent with  $c^{\lambda}=\frac{1}{\lambda}(\frac{W^{\lambda}}{2}-\beta_0^{\lambda})= 4\pi^2 (\frac{\kappa_0-1}{\kappa_0})$.
	
Now, it holds that 
	\begin{equation}
		\label{FM}
	F_M^{\lambda}\geq \beta_0^{\lambda}+c^{\lambda},
	\end{equation}
	\noindent otherwise $ F_M^{\lambda}< \left( \beta_0^{\lambda}+c^{\lambda}\right)$ and then the restriction on volume of $\Omega^{\lambda}$, eq. (\ref{Vrestriction}), is contradicted:
	$$v=Vol(\Omega^{\lambda})\leq  (2\pi{\lambda}^{-1}) F_M^{\lambda}<	2\pi \lambda^{-1} \left( \beta_0^{\lambda}+c^{\lambda}\right)$$
	
	 Similarly, it holds
	
		\begin{equation}
		\label{Fm}
		F_m^{\lambda}\leq W^{\lambda}-\beta_0^{\lambda}-c^{\lambda},
	\end{equation}
	\noindent otherwise 	$F_m^{\lambda}> W^{\lambda}-\beta_0^{\lambda}-c^{\lambda}$ and then  (\ref{Vrestriction}) is contradicted:
	
		$$v=Vol(\Omega^{\lambda})\geq  ( 2\pi{\lambda}^{-1}) F_m^{\lambda}> 2\pi{\lambda}^{-1}\left( W^{\lambda}-\beta_0^{\lambda}-c^{\lambda}\right)$$

We now claim  that $F_m^{\lambda}\geq \beta_0^{\lambda}$. Otherwise, $-F_m^{\lambda}> -\beta_0^{\lambda}$ and then, using (\ref{FM}), we estimate, by $(\ref{psi2})$:
	$$	I_{(S^2\times S^1 \times S^1,G_\lambda)}(v)  \geq  \int_{\delta E_{\lambda}}  2  dx_{\lambda}\geq2( F_M^{\lambda}-F_m^{\lambda})\geq 2 c^{\lambda}=8 \pi^2 \left(\frac{\kappa_0-1}{\kappa_0}\right)>4\pi^2, $$
	\noindent since $\kappa_0>2$. This contradicts (\ref{upper})  for  $\lambda\geq \lambda^*>2\pi$.

	Similarly, we have  $F_M^{\lambda}\leq W^{\lambda}-\beta_0^{\lambda}$.  Otherwise, $F_M^{\lambda}> (W^{\lambda}-\beta_0^{\lambda})$ and then, using (\ref{Fm}), we estimate, by $(\ref{psi2})$: 
	$$	I_{(S^2\times S^1 \times S^1,G_\lambda)}(v)  \geq  \int_{\delta E_{\lambda}}  2  dx_{\lambda} \geq 2( F_M^{\lambda}-F_m^{\lambda})\geq 2 c^{\lambda}=8 \pi^2 \left(\frac{\kappa_0-1}{\kappa_0}\right)>4\pi^2, $$
	\noindent since $\kappa_0>2$. This also contradicts (\ref{upper})  for $\lambda\geq \lambda^*>2\pi$.

\end{proof}

 
 We now prove Theorem \ref{Thm2}. In the following, the volume of a closed region $\Omega \subset M^n$ will mean $n-$dimensional Riemannian measure of $\Omega$ and  we will refer to it as $V^{n}(\Omega)$. On the other hand,  the area of a closed region $\Omega$ in $M^n$ will mean the $(n-1)-$dimensional Riemannian measure of the boundary $\partial \Omega$ and   we will denote it by $V^{n-1}(\partial \Omega)$.
 
 \begin{proof}(of Theorem \ref{Thm2})

Let $v_0\in (0,\frac{V}{2})$, and $v\in (v_0,V-v_0)$. Recall $V=V^4(S^2\times S^1 \times S^1, g_0+ds^2+dt^2)=16\pi^3$. Let $\kappa_0>2$ be such that $\frac{V}{k_0}=v_0$. Let $\lambda_{0}>0$ be such that 
$$\kappa_0=  \frac{\lambda_0^2}{4\pi^2}\ \  \psi_1\left(\frac{4\pi^2}{\lambda_0^2}\right)$$

By remark \ref{IS}, $\kappa_0>2$ implies $\lambda_0>2.05$ and  $I_{(S^2\times \re)} \left(\frac{4\pi^2}{\lambda_0^2}\right)= \psi_1\left(\frac{4\pi^2}{\lambda_0^2}\right)$. Let $\lambda^*>\max\{\lambda_0, 2\pi, \kappa_0-1\}$ and $\lambda>\lambda^*$.

Let $\Omega^{\lambda}$ be an isoperimetric region of $(S^2\times S^1 \times S^1,\lambda^2g_0+\lambda^{-2}dt^2+\lambda^{-2}ds^2)$ with volume $v$. We will prove that in this case, we can symmetrize $\Omega^{\lambda}$ to obtain a region of the type $ D^{\lambda}_R\times T^2_{\lambda}$, with volume $v$, such that $V^3(\partial( D^{\lambda}_R\times T^2_{\lambda}))=V^3(\partial \Omega^{\lambda})$.

  Let  $\beta_0^{\lambda}= \lambda \frac{4\pi^2}{\kappa_0}$.	We parametrize $S^1_{\lambda}$ by $[0,2\pi\lambda^{-1}]$, identifying $0$ and $2\pi\lambda^{-1}$. For $s\in  [0,2\pi{\lambda^{-1}}]$, consider the slices

$$\Omega_s^{\lambda}=\Omega^{\lambda}\cap\left( (S^2\times S^1)^{\lambda} \times  \{s\}\right)$$ 
	
	We use the function $F_{\Omega^{\lambda}}:[0,2\pi \lambda^{-1}]\rightarrow\re$, with  $F_{\Omega^{\lambda}}(s)=V^3(\Omega_s^{\lambda})$. Since $\Omega^{\lambda}$ is closed, $F_{\Omega^{\lambda}}(s)$ is continuous.

	  By Lemma  \ref{ineq}, the maximum $F_M^{\lambda}$ and minimum $F_m^{\lambda}$, of $F_{\Omega^{\lambda}}$ satisfy 
	  $$\beta_0^{\lambda}\leq F_m^{\lambda}\leq F_M^{\lambda}\leq W-\beta_0^{\lambda}. $$
	
This implies that for each $s\in [0,2\pi\lambda^{-1}]$, 	

	 $$\beta_0^{\lambda}\leq V^3(\Omega_s^{\lambda}) \leq W-\beta_0^{\lambda}. $$
	 
	 	Recall, from Lemma \ref{beta}, that in $(S^2\times S^1)^{\lambda}$, regions of the type $D^{\lambda}_R\times S^1_{\lambda}$, $R>0$, with volume $w^{\lambda}$,  $w^{\lambda}\in (\beta_0^{\lambda}, V-\beta_0^{\lambda})$, are isoperimetric regions. This means $I_{(S^2\times S^1)^{\lambda}}(w^{\lambda})=V^1(S^1_{\lambda}) V^{1} (\partial D^{\lambda}_R)$, for some $R>0$ such that $w^{\lambda}=V^1(S^1_{\lambda}) V^2(D^{\lambda}_R)$.

 We now symmetrize the isoperimetric region $\Omega^{\lambda}$ as in the proof of the Ros Product Theorem: we replace each slice $\Omega_s^{\lambda}\subset(S^2\times S^1)^{\lambda} \times \{s\}$ by  $ D^\lambda_{R(s)}\times S^1_{\lambda}\times \{s\}$, the product of $S^1_{\lambda}$ and a disk $D^{\lambda}_{R(s)}\subset S_{\lambda}^2=(S^2,\lambda^2g_0)$, such that  $V^1(S^1_{\lambda}) V^2(D_{R(s)}^{\lambda})=V^{3}(\Omega_s^{\lambda})$.

  We denote the new region in $ (S^2\times S^1 \times S^1)^{\lambda}$ by $(\Omega^{\lambda})^*$. 

Being $F_{\Omega^{\lambda}}(s)$ continuous and since for each $s$ we are using a region of the type $  D_{R(s)}^{\lambda}\times S_{\lambda}^1$, the region $(\Omega^{\lambda})^*$ is closed in $ (S^2\times S^1 \times S^1)^{\lambda}$. Note also that $V^{4} ((\Omega^{\lambda})^*)=V^{4} (\Omega^{\lambda})$, since each replacing slice is of corresponding volume. Moreover,  for each $s \in [0,2\pi \lambda^{-1}]$, since the replacing slice is isoperimetric, we have 

$$ V^{2} (\partial (D^\lambda_{R(s)}\times S^1_{\lambda}\times \{s\}))=I_{(S^2\times S^1)^{\lambda}}(V^{3} (\Omega_s^{\lambda}))\leq V^{2} (\partial \Omega_s^{\lambda}),$$
 
it follows from the Ros product Theorem that 
\begin{equation}
	\label{firstsym}
	V^{3} (\partial(\Omega^{\lambda})^*)\leq V^{3} (\partial\Omega^{\lambda}).
\end{equation}

We now symmetrize $(\Omega^{\lambda})^*\subset (S^2\times S^1 \times S^1)^{\lambda}$ with respect to the other factor, $(S^1,\lambda^{-2}dt^2)$. We parametrize  $S^1_{\lambda}$ by $[0,2 \pi \lambda^{-1}]$, identifying $0$ and $2\pi\lambda^{-1}$. Consider the slices $\Omega^*_t$, $t\in [0,2\pi \lambda^{-1}]$:

$$(\Omega^{\lambda})^*_t=(\Omega^{\lambda})^* \cap(   S_{\lambda}^2\times  \{t\}\times S^1_{\lambda}).$$

For each slice we may compute again its $3$-volume with the function $H_{(\Omega^{\lambda})^*}:[0,2 \pi {\lambda}^{-1}]\rightarrow \re$, now with respect to the other factor $S^1_{\lambda}$:
$H_{(\Omega^{\lambda})^*}(t)=V^{3}((\Omega^{\lambda})^*_t)$ for $t\in [0,2 \pi \lambda^{-1}]$. Note that, by construction,  for each $s\in [0,2\pi \lambda^{-1})$ and for any  $t_1, t_2 \in [0,2 \pi \lambda^{-1}]$,   we have

$$ D_{R(s)}\times \{t_1\}\times\{s\}= D_{R(s)}\times \{t_2\}\times\{s\}.$$

\noindent This implies that both slices $(\Omega^{\lambda})^*_{t_1}$ and $(\Omega^{\lambda})^*_{t_2}$ have the same volume.   It follows that $H_{(\Omega^{\lambda})^*}(t)$ is constant.

We now claim that $\beta_0^{\lambda}\leq H_{(\Omega^{\lambda})^*}(t)\leq W^{\lambda}-\beta_0^{\lambda}$.  If the claim were not true, then $H_{(\Omega^{\lambda})^*}(t)< \beta_0^{\lambda}$ or $H_{(\Omega^{\lambda})^*}(t)> W^{\lambda}- \beta_0^{\lambda}$. And  we would have
$$V^{4}((\Omega^{\lambda})^*)=V^1(S^1_{\lambda})H_{(\Omega^{\lambda})^*}(t)<2\pi\lambda^{-1}\beta_0^{\lambda}= \frac{8\pi^3}{\kappa_0}=\frac{1}{2}\frac{V}{\kappa_0}=\frac{v_0}{2},$$

\noindent or 
$$V^{4}((\Omega^{\lambda})^*)=V^1(S^1_{\lambda})H_{(\Omega^{\lambda})^*}(t)>2\pi\lambda^{-1}W^{\lambda}-2\pi\lambda^{-1}\beta_0^{\lambda}=V-\frac{1}{2}\frac{V}{\kappa_0}=V-\frac{v_0}{2},$$

\noindent which are ruled out by the hypothesis $v\in (v_0,V-v_0)$, since  $V^{4} ((\Omega^{\lambda})^*)=V^{4} (\Omega^{\lambda})=v.$

We then construct a new region $((\Omega^{\lambda})^*)^*\subset (S^2\times S^1\times S^1)^{\lambda}$ by letting each slice $(\Omega^{\lambda})^{*}_{t}$ be replaced by $ D_{R_0}^{\lambda}\times \{t\}\times S^1_{\lambda}$, where $R_0$ is such that  $V^1(S^1_{\lambda}) V^2(D_{R_0}^{\lambda})=H_{(\Omega^{\lambda})^*}(t)$.

Since  $H_{(\Omega^{\lambda})^*}(t)$  is constant, then $R_0$ is constant, and we get

 \begin{equation}
 	\label{vol}
 	V^{4}(((\Omega^{\lambda})^*)^*)=V^{4}((\Omega^{\lambda})^*)=V^{4}(\Omega^{\lambda})=v
 \end{equation}

 Moreover, the region $((\Omega^{\lambda})^*)^*$ is closed. And  since $\beta_0^{\lambda}\leq H_{(\Omega^{\lambda})^*}\leq W^{\lambda}- \beta_0^{\lambda}$, each replacing slice is isoperimetric. Then we have

$$V^{2}(\partial( D_{R_0}^{\lambda}\times  \{t\} \times S^1_{\lambda} ))=I_{(S^2\times S^1)^{\lambda}}(V^{3} (\Omega^{\lambda})^{*}_t) \leq V^{2}(\partial(\Omega^{\lambda})^{*}_t).$$

\noindent Hence,   using the Ros Product Theorem we get $V^{3}(\partial ((\Omega^{\lambda})^*)^*)\leq V^{3}(\partial(\Omega^{\lambda})^*)$. And, since $\Omega^{\lambda}$ was already an isoperimetric region, together with eq. (\ref{firstsym}):

\begin{equation}
	\label{area}
	V^{3}(\partial ((\Omega^{\lambda})^{*})^*)\leq V^{3}(\partial(\Omega^{\lambda})^*)\leq V^{3}(\partial\Omega^{\lambda})=I_{(S^2\times S^1 \times S^1)^{\lambda}}(v).
	\end{equation}

 Finally, by construction, since $R_0$ is constant, we have that in fact
$$((\Omega^{\lambda})^*)^*= D_{R_0}^{\lambda}\times  S^1_{\lambda}\times S^1_{\lambda} .$$

We conclude by  (\ref{vol}) and (\ref{area})  that $ D_{R_0}^{\lambda}\times S^1_{\lambda}\times S^1_{\lambda}$ is an isoperimetric region.

\end{proof}

\end{document}